\documentclass[12pt]{article}
\usepackage{amsmath,amssymb,amsthm}
\usepackage{enumerate, cite}
\usepackage{graphicx}
\usepackage[dvips]{color}

\newtheorem{theorem}{Theorem}[section]
\newtheorem{proposition}[theorem]{Proposition}
 \newtheorem{example}{Example}[section]
\newtheorem{claim}{Claim}
\newtheorem{case}{Case}

\newtheorem{problem}[example]{Problem}

\newcommand{\gp}{{\rm gp}}
\newcommand{\Z}{{\rm Z}}
\newcommand{\PP}{{\rm P}}

\begin{document}

\title{Zero forcing number versus general position number in tree-like graphs}

\author{
Hongbo Hua $^{a}$\,\thanks{corresponding author}
\and
Xinying Hua $^{b}$
\and
Sandi Klav\v zar $^{c,d,e}$
}

\date{\today}

\maketitle

\begin{center}
	$^a$ Faculty of Mathematics and Physics, Huaiyin Institute of Technology \\
             Huai'an, Jiangsu 223003, PR China \\
             \texttt{hongbo\_hua@163.com} \\
	\medskip

	$^b$ College of Science, Nanjing University of Aeronautics \& Astronautics \\
	Nanjing, Jiangsu 210016, PR China \\
	\texttt{xyhuamath@163.com} \\
	\medskip

	$^c$ Faculty of Mathematics and Physics, University of Ljubljana, Slovenia\\
	\texttt{sandi.klavzar@fmf.uni-lj.si}\\
	\medskip
	
	$^d$ Faculty of Natural Sciences and Mathematics, University of Maribor, Slovenia\\
	\medskip
	
	$^e$ Institute of Mathematics, Physics and Mechanics, Ljubljana, Slovenia\\
	\medskip
\end{center}

\begin{abstract}
Let $\Z(G)$ and $\gp(G)$ be the zero forcing number and the general position number of a graph $G$, respectively. Known results imply that $\gp(T)\ge \Z(T) + 1$ holds for every nontrivial tree $T$. It is proved that the result extends to block graphs. For connected, unicyclic graphs $G$ it is proved that $\gp(G) \ge \Z(G)$. The result extends neither to bicyclic graphs nor to quasi-trees. Nevertheless, a large class of quasi-trees is found for which $\gp(G) \ge \Z(G)$ holds.
\end{abstract}


\noindent
AMS Subj.\ Class.\ (2020): 05C69, 05C12

\section{Introduction}

In linear algebra, the zero forcing number of a graph was introduced in~\cite{SGWGroup} to bound the minimum rank of matrices associated with graphs. In physics, the zero forcing was introduced to study controllability of quantum systems~\cite{Burgarth}; in computer science, it appears as the fast-mixed search model for some pursuit-evasion games~\cite{Yang}; in network science, it models the spread of a disease over a population~\cite{Dreyer}.

Since its introduction by the  ``AIM group" in~\cite{SGWGroup}, the zero forcing number has become a graph parameter being widely investigated for its own sake. In 2008, Aazami \cite{Aazami} proved the NP-hardness of computing the zero forcing number of a graph. So, it makes sense to establish sharp bounds on the zero forcing number for general graphs and to derive formulas for special graphs, see  \cite{Ferrero,Javaid,Kang1,Lu,Oboudi} for a selection of relevant results.

The general position number of a graph was introduced in~\cite{Manuel}. A couple of years earlier, however, the invariant was in different terminology  considered in~\cite{UllasChandran}. Moreover, in the special case of hypercubes it was much earlier studied in~\cite{korner-1995}. In~\cite{Anand}, general position sets in graphs were characterized. Several additional papers on the concept followed, many of them dealing with bounds on  the general position number and exact results in product graphs, Kneser graphs, and more, see~\cite{Ghorbani, Klavzar2, Klavzar1, neethu-2021, Patkos, thomas-2020, Tian, tian-2021}. In addition, the concept was very recently extended to the Steiner general position number~\cite{Klavzar-Kuziak}. 

Motivated by the comparative results between the zero forcing number and one of the central concepts of metric graph theory, the metric dimension, focusing on trees and unicyclic graphs~\cite{Eroh1, Eroh2}, we consider here the relation between the zero forcing number and the general position number. Now, from~\cite[Theorem 2]{Oboudi} we know that if $T$ is a tree on at least two vertices, then $\Z(T)\leq \ell(T)-1$, where $\ell(T)$ is the number of leaves of $T$. On the other hand, it was observed in~\cite[Corollary 3.7]{Manuel} that $\gp(T ) = \ell(T)$. Hence, if $T$ is a tree on at least two vertices, then
\begin{equation}
\label{eq:starting-point}
\gp(T )\geq \Z(T)+1\,.
\end{equation}
This relation prompted us to investigate whether there are additional larger families for which the zero forcing number is a lower bound for the general position number. We proceed as follows. In the next subsection the concepts studied are formally introduced and additional definitions stated. In Section~\ref{sec:unicyclic} we prove that if $G$ is a connected, unicyclic graphs $G$, then $\gp(G) \ge \Z(G)$. We also demonstrate that the inequality does not extend to bicyclic graphs. In Section~\ref{sec:block-graphs-quasi-trees} we first prove that~\eqref{eq:starting-point} holds for arbitrary block graphs. Then we show that the zero forcing number and the general position number are in general not related on quasi-trees. On the other hand, a large class of quasi-trees is found for which $\gp(G) \ge \Z(G)$ holds. We conclude the paper with three open problems.

\subsection{Definitions}

The order and the size of a graph $G = (V(G), E(G))$ will be respectively denoted by $n(G)$ and $m(G)$. Let $G$ be a connected graph. If $m(G) = n(G)-1$, then $G$ is a {\em tree}, if $m(G)=n(G)$, then $G$ is a {\em unicyclic graph}, and if $m(G)=n(G)+1$, then $G$ is a {\em bicyclic graph}. If  $G$ contains a vertex $v$, such that $G-v$ is a tree, then $G$ is a {\em quasi-tree}, the vertex $v$ is a {\em quasi-vertex} of $G$. A connected graph $G$ is a {\em block graph} if each 2-connected component of $G$ is a clique. Let $\mathcal{L}(G)$ denote the set of pendent vertices of $G$, so that $\ell(G)=|\mathcal{L}(G)|$.

For a graph $G$, assume that all its vertices are given one of two colors,  black and white by convention. Let $S$ denote the (initial) set of
black vertices of $G$. The {\em color-change rule} changes the color of a vertex from white to black if the white vertex $y$ is the only white neighbor of a black vertex $x$, and we say that $x$ forces $y$. Obviously, at each step of the color change, there may be two or more vertices capable of forcing the same vertex. The \emph{zero forcing number} $\Z(G)$ of $G$ is the minimum cardinality of a set $S$ of black vertices (while all vertices of $V(G)\setminus S$ are colored white) such that all vertices of $V(G)$ are turned black after finitely many applications of the color-change rule.

The distance $d_{G}(u, v)$ is the length of a shortest $u,v$-path in $G$. The interval $I_{G}(u, v)$ between vertices $u$ and $v$ is a vertex subset which consists of all vertices lying on shortest $u,v$-paths. A vertex subset $R$ of a graph $G$ is a {\em general position set} if no three vertices from $R$ lie on a common shortest path. The \emph{general position number} ($\gp$-number for short) $\gp(G)$ of $G$ is the number of vertices in a largest general position set of $G$. For convenience, we say that a largest general position set is a $\gp$-set.

For a positive integer $k$ we will use the notation $[k] = \{1,\ldots, k\}$.

\section{Unicyclic graphs}
\label{sec:unicyclic}

In this section we prove a result parallel to~\eqref{eq:starting-point} for unicyclic graphs. Before stating and proving the result, we demonstrate that it cannot be extended to bicyclic graphs.

Let $H_1$ be the top bicyclic graph from Fig.~\ref{fig:bicyclic}, and let $H_2$ be the bottom bicyclic graph from  the same figure.  $H_1$ actually represents a two-parametric family of graphs, but we will assume that $s$ and $t$ are fixed and denote the representative simply by $H_1$.

\begin{figure}[ht!]
\begin{center}
\includegraphics*[width=10.5cm]{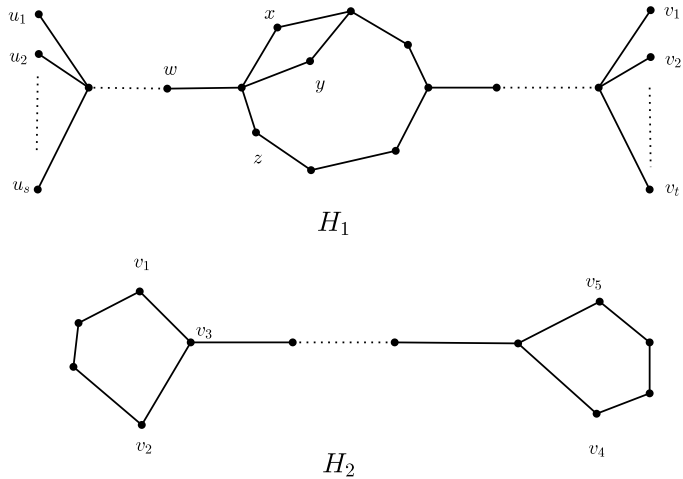}
\end{center}
\caption{Bicyclic graphs $H_{1}$ and $H_{2}$}
\label{fig:bicyclic}
\end{figure}

The set $\{u_{1},\ldots, u_{s}, x, y, v_{1}, \ldots, v_{t-1}\}$  is a minimum zero forcing set of $H_{1}$. Also, it can be seen that for $2\leq s+t\leq 3$, the set  $\{x, y, z, w\}$ is a gp-set  of $H_{1}$, while for  $ s+t\geq 4$, the set $\{u_{1}, \ldots, u_{s}, v_{1}, \ldots, v_{t}\}$  is a gp-set  of $H_{1}$. Thus, if $s+t\geq 4$,  then $\Z(H_{1})=s+t+1>s+t = \gp(H_{1})$, and if $2\leq s+t\leq 3$,  then $\Z(H_{1})=s+t+1\leq 4=|\{x, y, z, w\}| = \gp(H_{1})$. On the other hand, $\{v_{1}, v_{3}, v_{5}\}$  is a minimum zero forcing set of $H_{2}$ and $\{v_{1}, v_{2}, v_{4}, v_{5}\}$  is a gp-set  of $H_{2}$. Thus, $\Z(H_{2}) = 3 < 4 = \gp(H_{2})$.

We have thus seen that the zero forcing number and the general position number are incomparable on bicyclic graphs. On the other hand, the main result of this section asserts that the situation is different for unicyclic graphs.

\begin{theorem}
\label{thm:unicyclic}
If $G$ is a connected, unicyclic graph, then $\gp(G)\geq\Z(G)$.
\end{theorem}

The rest of the section is devoted to the demonstration of Theorem~\ref{thm:unicyclic}.

The \emph{path cover number} $\PP(G)$  of a graph $G$ is the smallest positive integer $k$ such that there are $k$ vertex-disjoint induced paths in $G$ such that every vertex of $G$ is a vertex of one of the paths. It was proved in~\cite{Hogben} that $\PP(G)\leq \Z(G)$ holds for each graph $G$. For   unicyclic graphs, Row proved the following stronger result.

\begin{theorem}{\rm \cite[Theorem 4.6]{Row}}
\label{th0}
If $G$ is a connected unicyclic graph, then $\Z(G)=\PP(G)$.
\end{theorem}

For the proof of Theorem~\ref{thm:unicyclic} we need to recall several concepts and results from~\cite{Barioli, Row}.

Let $G$ be a graph and $x$ a vertex of $G$. If $G-x$ has at least two components  which are paths, each joined to $x$ in $G$ at only one endpoint, then vertex $x$ is called \emph{appropriate}. A vertex $x$ is called a \emph{peripheral leaf} if $x$ is adjacent to
only one other vertex $y$, and $y$ is adjacent to no more than two vertices. The \emph{trimmed form} $\breve{G}$ of a graph $G$ is an induced subgraph of $G$ obtained by a sequence of deletions of appropriate vertices, isolated paths, and peripheral leaves until no more such deletions are possible. Barioli, Fallet, and Hogben~\cite{Barioli} proved that $\breve{G}$ is unique. If $\breve{G}$ is obtained from $G$ by performing $n_1$ deletions of appropriate vertices, $n_2$ deletions of isolated paths, and $n_3$ deletions of peripheral leaves, then $\PP(G) = \PP(\breve{G}) + n_2-n_1$~\cite{Barioli}.

Let $C_n$ be an $n$-cycle and let $U\subseteq V(C_n)$.  The graph $H$ obtained from $C_n$ by appending a leaf to each vertex of $U$ is called a \emph{partial sun}. The term \emph{segment} of $H$ will refer to any maximal subset of consecutive vertices in $U$. The segments of $H$ will be denoted $U_{1}, \ldots, U_{t}$. For a partial sun $H$ with segments $U_{1}, \ldots, U_{t}$, it was proved in~\cite{Barioli} that $\PP(H) = \max\{2, \sum_{i=1}^{t}\lceil \frac{|U_{i}|}{2}\rceil\}$. The trimmed form of a unicyclic graph $G$ is either the empty graph or a partial sun \cite{Barioli}. For an example see~Fig.~\ref{fig:trimmed} and note that $\breve{G}$ is a partial sun.

\begin{figure}[ht!]
\begin{center}
\includegraphics*[width=12cm]{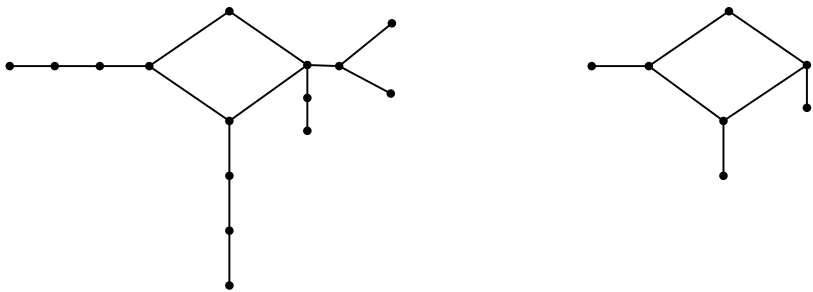}
\end{center}
\vspace*{-0.7cm}
\caption{A unicyclic graph $G$ (left) and its trimmed form $\breve{G}$ (right)}
\label{fig:trimmed}
\end{figure}

The following properties of unicyclic graphs with respect to  their trimmed graphs will enable us to derive Theorem~\ref{thm:unicyclic}.

\begin{theorem}\label{th001}
Let $G$ be a connected, unicyclic graph and let $\breve{G}$, $U$, $n_{1}$, and $n_{2}$ be defined as above. Then the following hold.
\begin{enumerate}
\item[(i)] If $\breve{G}$ is a partial sun, then $\gp(G)\geq \max\{2, |U|\}+n_{2}-n_{1}$.
\item[(ii)] If $\breve{G}$ is the empty graph, then $n_{2}-n_{1}\leq \ell(G)$.
\end{enumerate}
\end{theorem}

\begin{proof}
(i)
Since $\breve{G}$ is a partial sun, $G$ is not a cycle. Let $C_{l}$ be the unique cycle of $G$.

For any branch vertex $v$ on the cycle $C_{l}$, we denote by $T_{G}(v)$ the subtree containing $v$ of $G-\{u, w\}$, where $u\in N_{C_{l}}(v)$ and $w\in N_{C_{l}}(v)$. Such a subtree $T_{G}(v)$ is called a \emph{root tree} at $v$. Let $U_{1}, \ldots, U_{t}$ be the segments of $\breve{G}$.

\begin{claim}\label{clm01}
If $G$ has no appropriate vertices and $x$ is a vertex of $V(\breve{G})$ with $d_{\breve{G}}(x)=2$, then $d_{G}(x)=2$.
\end{claim}
\begin{proof}
Suppose to the contrary
that $d_{G}(x)\geq 3$. Since $G$ has no appropriate vertices, $x$ is a branch vertex on $C_{l}$. Since $G$ has no appropriate vertices,  $T_{G}(x)$ is a path with $x$ being one end-vertex. So, we need to turn $x$ into a 2-degree vertex in $\breve{G}$ by a sequence of deletions of appropriate vertices, isolated paths, and peripheral leaves.
By our assumption that $G$ has no appropriate vertices, we can only use the trimmed operation on $T_{G}(x)$ by repeatedly deleting peripheral leaves such that $x$ is turned into a 2-degree vertex. This is impossible by the definition of  peripheral leaves. Thus, $d_{\breve{G}}(x)=3$, a contradiction.
\end{proof}

We now distinguish the following two cases.

\begin{case}
$G$ has no appropriate vertices.
\end{case}
\noindent
In this case $n_{1}=n_{2}=0$. Moreover, for any branch vertex $v$ on $C_{l}$, $T_{G}(v)$ is a path with $v$ being its one end-vertex. Thus, $\ell(G)=\sum_{i=1}^{t}|U_{i}|=|U|$.

When $|U|=1$, by Claim~\ref{clm01} and our assumption that $G$ has no appropriate vertices, $G$ has only one branch vertex, say $v$, on $C_{l}$. Let $u$ and $w$ be two neighbors of $v$ on $C_{l}$. Clearly,  $U=\{v\}$. Note that in the current case, $\breve{G}$ can only be obtained from $G$ by a sequence of deletions of  peripheral leaves. Then  $\ell(G)=1$. So, $\{u, w\}\cup \mathcal{L}(G)$ forms a general position set.
Thus, $\gp(G)\geq |\{u, w\}\cup \mathcal{L}(G)|=3>2=\max\{2, |U|\}+n_{2}-n_{1}$. Now, we assume that $|U|\geq 2$. Then  the set of all pendent vertices of $G$ forms a general position set. So, $\gp(G)\geq \ell(G)=|U|= \max\{2, |U|\}+n_{2}-n_{1}$.

\begin{case}
$G$ has at least one appropriate vertex.
\end{case}
\noindent
In this case $n_{1}\geq 1$ and $n_{2}\geq 2n_{1}$.
Since $\breve{G}$ is a partial sun, $C_{l}$ does not contain appropriate vertices. So, all appropriate vertices of $G$ belong to the set $V(G)\setminus V(C_{l})$. Let  $B(G)$ be the set of branch vertices of $G$.

First, we assume that for
any vertex $v$ of $U$, the subtree $T_{G}(v)$ does not contain an appropriate vertex. Hence any such subtree $T_{G}(v)$ is a path with $v$ being one
end-vertex, and then $\ell(\breve{G})=|U|$. Moreover, since $\breve{G}$ is a partial sun, all appropriate vertices  are contained
in $\bigcup_{x\in B(G)\setminus U} (V(T_{G}(x))\setminus\{x\})$.
Then $|\mathcal{L}(G)\setminus\mathcal{L}(\breve{G})|\geq n_{2}$, that is, $\ell(G)\geq n_{2}+\ell(\breve{G})=n_{2}+|U|$. Since  the set of all pendent vertices of $G$ forms a general position set of $G$, $\gp(G)\geq \ell(G)\geq n_{2}+|U|$. If $|U|=1$, as
$n_{1}\geq 1$, then  $\gp(G)\geq n_{2}+1\geq \max\{2, |U|\}+n_{2}-n_{1}$.
If $|U|\geq 2$, as
$n_{1}\geq 1$, then  $\gp(G)\geq \ell(G)\geq n_{2}+|U|>\max\{2, |U|\}+n_{2}-n_{1}$.

Assume that $U$ has $s\in [|U|]$ vertices each of whose root trees contains at least one appropriate vertex. Since for some branch vertex $v\in V(C_{l})\setminus U$, the set $V(T_{G}(v))\setminus \{v\}$ may contain appropriate vertices, we have $s\leq n_{1}$. Thus, $\ell(G)\geq n_{2}+(|U|-s)\cdot 1\geq |U|+n_{2}-n_{1}$.

If $|U|\geq 2$, then because the set of all pendent vertices of $G$ forms a general position set of $G$, we get $\gp(G)\geq \ell(G)\geq |U|+n_{2}-n_{1}= \max\{2, |U|\}+n_{2}-n_{1}$.

Suppose now that $|U|=1$. Let $U=\{v\}$, and let $N(v) \cap V(C_l) = \{u, w\}$. Then $T_{G}(v)$  contain at least one appropriate vertex. We first prove the following claim.

\begin{claim}\label{clm02}
Let $G$ be a unicyclic graph with $\breve{G}$ being a partial sun. If the unique cycle $C_{l}$ of $G$ has $q$ branch vertices, say $v_{1},\,\ldots,\,v_{q}\,(q\geq 1)$, such that each $V(T_{G}(v_{i}))\setminus \{v_{i}\}$ ($i\in [q]$) contains appropriate vertices, then $\ell(G)\geq n_{2}-n_{1}+q$.
\end{claim}
\begin{proof}
Since $\breve{G}$ is a partial-sun, we suppose as before that  $G$ can be reduced to $\breve{G}$  by performing $n_1$ deletions of appropriate vertices outside $C_{l}$, $n_2$ deletions of isolated paths outside $C_{l}$, and $n_3$ deletions of peripheral leaves outside $C_{l}$. Note that each step of deletion of an old appropriate vertex and the corresponding isolated paths from $G$ will produce at most one new appropriate vertex in the resulting subgraph of $G$. Moreover, if a new appropriate vertex is born with the process of deletion of an old appropriate vertex in $G$, then a new pendent path must be produced in the resulting subgraph at the same time (this new pendent path becomes
a new isolated path in future). Hence, if there are  $p$ new appropriate vertices produced during the process of trimming $G$,  then $G$ has at least $n_{2}-p$ pendent vertices. Since  each $V(T_{G}(v_{i}))\setminus \{v_{i}\}$ ($i\in [q]$) contains at least one appropriate vertex, we have $p\leq n_{1}-q$.   Therefore $\ell(G)\geq n_{2}-p\geq n_{2}-n_{1}+q$.
\end{proof}

If $q\geq 2$, then since the set of all pendent vertices of $G$ forms a general position set, we get $\gp(G)\geq \ell(G)\geq n_{2}-n_{1}+2=\max\{2, |U|\}+n_{2}-n_{1}$ by Claim \ref{clm02} and our assumption that $|U|=1$. Assume hence that $q=1$. Then all appropriate vertices of $G$ belongs to $T_{G}(v)$. The fact that $|U|=1$  together with Claim~\ref{clm01} yields that $G$ has $v$ as its unique branch vertex.
Thus, $\{u, w\}\cup \mathcal{L}(G)$ forms a general position set which in turn implies that $\gp(G)\geq |\{u, w\}\cup \mathcal{L}(G)|= \ell(G)+2\geq (n_{2}-n_{1}+1)+2>n_{2}-n_{1}+2=\max\{2, |U|\}+n_{2}-n_{1}$ by Claim \ref{clm02}. This proves (i).

\medskip
(ii) Assume that $G$ can be reduced to $\breve{G}$  by performing $n_1$ deletions of appropriate vertices, $n_2$ deletions of isolated paths, and $n_3$ deletions of peripheral leaves. Since $\breve{G}$ is the empty graph, $G$ has at least one appropriate vertex, that is, $n_{1}\geq 1$. Let $C_{l}$ be the unique cycle in $G$.  We proceed by induction on $n(G)+n_{1}$.

If $n_{1}=1$, then since $\breve{G}$ is the empty graph, the unique appropriate vertex, say $v$, must lie on the cycle $C_{l}$. Moreover, as $n_{1}=1$, the deletion of $v$ results in only isolated paths. Among all isolated paths of $G-v$, there is at most one isolated path whose two end-vertices are not pendent vertices of $G$. So, $n_{2}\leq \ell(G)+1$. Thus, $n_{2}-n_{1}\leq (\ell(G)+1)-1=\ell(G)$.

Let now $k>n(G)+1$. Assume that $n'_2-n'_1 \leq  \ell(G')$ holds for all unicyclic graphs $G'$ with $n(G')+n'_{1}<k $ and $\breve{G}'$ being the empty graph. Let $G$ be an unicyclic graph, with $n(G)+n_{1}=k$, which can be reduced to the empty graph $\breve{G}$  by performing $n_{1}$ deletions of appropriate vertices.

Assume first that there exists an appropriate vertex, say $v$, which lies outside $C_{l}$. Perform one step of the deletion of the appropriate vertex $v$, and the deletion of isolated paths in $G-v$ corresponding to $v$, and denote the resulting unicyclic graph by $G'$. Clearly, $\breve{G}'$ is also the empty graph. Assume that $G'$ can be reduced to  $\breve{G}'$ by performing $n'_1$ deletions of appropriate vertices, $n'_2$ deletions of isolated paths, and $n'_3$ deletions of peripheral leaves. Then $n'_{1}=n_{1}-1$. As $n(G')+n'_{1}<k $, by the induction hypothesis, $n'_{2}-n'_{1}\leq \ell(G')$ holds for  $G'$. Since $v$ is an appropriate vertex outside the cycle, one step of the deletion of $v$ and the corresponding isolated paths will produce at most one new pendent path in $G'$. Assume that there are $t$ pendent paths attaching to $v$ in $G$. Then $n_{2}= n'_{2}+t$. So, $n_{2}-n_{1}=  (n'_{2}+t)-(n'_{1}+1)=(n'_{2}-n'_{1})+(t-1)\leq \ell(G')+(t-1)\leq \ell(G)$.

Assume second that all the appropriate vertices of $G$ lie on the cycle $C_{l}$. Let $v$ be an arbitrary  appropriate vertex. Set $G'=G[V(G)\setminus (V(T_{G}(v))\setminus \{v\}$)]. Then $G'$ is a unicyclic graph whose unique cycle is still $C_{l}$ and $d_{G'}(v)=2$. Obviously, $\breve{G}'$ is also the empty graph. Assume that $G'$ can be reduced to  $\breve{G}'$ by performing $n'_1$ deletions of appropriate vertices, $n'_2$ deletions of isolated paths, and $n'_3$ deletions of peripheral leaves. Then $n'_{1}= n_{1}$ or $n'_{1}= n_{1}-1$. Since $n(G')+n'_{1}<k $, by the induction hypothesis, $n'_{2}-n'_{1}\leq \ell(G')$. Assume that $v$ is attached to $t$  pendent paths in $G$. By the construction of $G'$ and our assumption that $v$ is an appropriate vertex lying on $C_{l}$, we infer that $n_{2}\leq n'_{2}+t$.

So, if $n'_{1}= n_{1}-1$, then $n_{2}-n_{1}\leq  (n'_{2}+t)-(n'_{1}+1)=(n'_{2}-n'_{1})+(t-1)< \ell(G')+t= \ell(G)$, and if $n'_{1}= n_{1}$, then $n_{2}-n_{1}\leq  (n'_{2}+t)-n'_{1}=(n'_{2}-n'_{1})+t\leq\ell(G')+t= \ell(G)$.
\end{proof}

Now all is ready to prove Theorem~\ref{thm:unicyclic}. If $G$ is a cycle graph, then $\gp(G)\geq 2=\Z(G)$. Assume in the rest that $G$ is not a cycle. Let $\breve{G}$ be obtained from $G$ by a sequence of $n_1$ appropriate vertex deletions, $n_2$ isolated path deletions, and $n_3$ peripheral leaf deletions. Recall that $\breve{G}$ is either the empty graph or a partial sun, and consider the following two cases.

Assume that $\breve{G}$ is a partial sun and let $U_{1}, \ldots, U_{t}$ be the segments of $\breve{G}$ with $\sum_{i=1}^{t}|U_{i}|=|U|$.  From~\cite{Barioli}, we know that $\PP(G)=\PP(\breve{G})+n_{2}-n_{1}$ and $\PP(\breve{G})= \max\{2, \sum_{i=1}^{t} \lceil \frac{|U_{i}|}{2}\rceil\}$. Since $|U_{i}|\geq 1$, we have $\sum_{i=1}^{t}\lceil \frac{|U_{i}|}{2}\rceil\leq |U|$. By Theorem \ref{th001}(i), we have  $\gp(G)\geq \max\{2, |U|\}+n_{2}-n_{1}\geq \max\{2, \sum_{i=1}^{t}\lceil \frac{|U_{i}|}{2}\rceil\}+n_{2}-n_{1}=\PP(\breve{G})+n_{2}-n_{1}=\PP(G)$. Combining this fact with Theorem~\ref{th0} gives $\gp(G)\geq \Z(G)$.

Assume second that $\breve{G}$ is the empty graph. By the proof of~\cite[Theorem 4.6]{Row} we have $\Z(G)=\Z(\breve{G})+n_{2}-n_{1}=n_{2}-n_{1}$. Further, by Theorem~\ref{th001}(ii) we have $\Z(G)=n_{2}-n_{1}\leq \ell(G)$. On the other hand, since  the set of all pendent vertices of $G$ forms a general position set, $\gp(G)\geq \ell(G)$. Therefore, $\gp(G)\geq \Z(G)$.

\section{Block graphs and quasi-trees}
\label{sec:block-graphs-quasi-trees}

In this section we consider two broad generalizations of trees---block graphs and quasi-trees---and relate them to~\eqref{eq:starting-point}. In the main result of the section we prove that~\eqref{eq:starting-point} extends to all block graph. Then we demonstrate that the zero forcing number and the general position number are not comparable on quasi-trees. On the positive side we show that $\gp(G)\geq \Z(G)+1$ still holds for a rich class of quasi-trees $G$. We conclude the section by showing that~\eqref{eq:starting-point} naturally extends to forests.

Before proving the result for block graphs, some preparation is needed. A vertex of a graph is \emph{simplicial} if its neighbours induce a complete subgraph. A block in a graph is said to a \emph{pendent block} if it has exactly one cut vertex. Note that all vertices but one of a pendent block are simplicial vertices.

\begin{theorem}\label{th3}
If $G$ is a block graph with $n(G)\ge 2$, then $\gp(G)\geq \Z(G)+1$.
\end{theorem}

\begin{proof}
Let $S$ be the set of simplicial vertices of a block graph $G$. Then it was proved in~\cite[Theorem 3.6]{Manuel} that $S$ is a general position set of $G$ and that $\gp(G) = |S|$. To prove the theorem it thus suffices to show that $\Z(G)\leq |S|-1$.

Suppose that $G$ has $k$ cut vertices and set $n = n(G)$. Then $|S|=n-k$. We prove that $\Z(G)\leq n-k-1$ by induction on $n+k$. Clearly, $n+k\geq n$. if $n+k=n$, that is, if $k=0$, then $G\cong K_{n}$ and $\Z(G)=n-1$, as desired. Suppose next that $k\geq 1$ and let $B$ be a pendent block of $G$ sharing the unique cut vertex $v_{0}$ with a smaller block graph $G'$ of order $n'$ and with $k'$ cut vertices. Note that $n' = n - |B| + 1$. Since $n'+k' < n + k$, the induction hypothesis implies $\Z(G')\leq n'-k'-1$.

Let $S'$ be a zero forcing set of $G'$ with $|S'| = \Z(G')$. Let $V(B)=\{v_{0}, v_{1}, \ldots, v_{t}\}$. Then $t = |B|-1$. If $t=1$, then $S'$ is also a zero forcing set of $G$. Assume next that $t\geq 2$. Let $S = S' \cup \{v_{1}, \ldots, v_{t-1}\}$. Note that each of the vertices $v_{1}, \ldots, v_{t-1}$ is simplicial. Moreover, all vertices of $G'$ are forced to be black under $S'$. Thus, $v_{t}$ is the unique white neighbor of $v_{0}$ in $G$. Then $v_{t}$ is forced to be black by $v_{0}$  under $S$ in $G$. So, $S$ is a zero forcing set of $G$. In conclusion, if $t=1$, then
$$\Z(G)\leq |S|=|S'|\leq n'-k'-1=(n-1)-k'-1\leq n-k-1\,,$$
and if $t\geq 2$, then since $k'+1\geq k$,
$$\Z(G)\leq |S|=|S'|+t-1\leq n'-k'-1+t-1=n-k'-2\leq n-k-1\,,$$ and we are done.
\end{proof}

We next demonstrate that the zero forcing number and the general position number are not comparable on quasi-trees. For this sake consider the quasi-trees $H_{3}$ and $H_{4}$ from Fig.~\ref{fig:quasi-trees}. Just as $H_1$ in Fig.~\ref{fig:bicyclic}, also $H_3$ and $H_4$  each represents an infinite family of graphs, but we consider their parameters as fixed and denote the representatives simply by $H_3$ and $H_4$.

\begin{figure}[ht!]
\begin{center}
\includegraphics*[width=11cm]{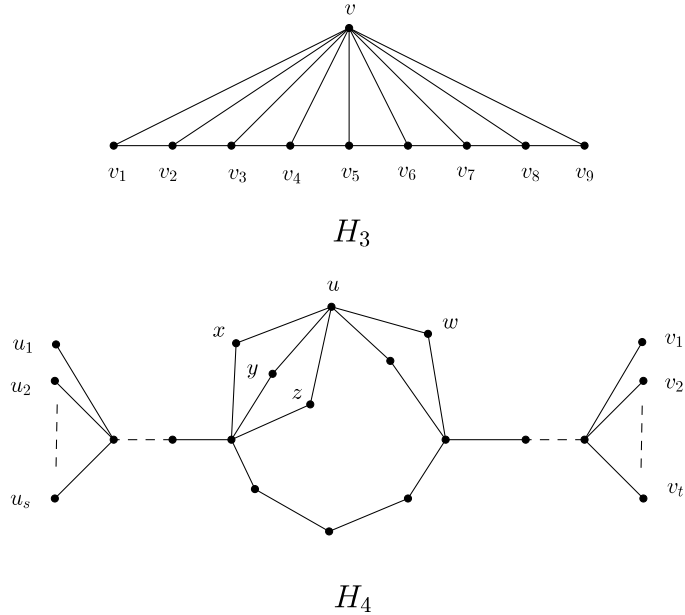}
\end{center}
\vspace{-8mm}
\caption{Quasi-trees $H_{3}$ and $H_{4}$}
\label{fig:quasi-trees}
\end{figure}

It is straightforward to verify that $\{v_{1}, v\}$ is a minimum zero forcing set of $H_{3}$, and that $\{v_{1}, v_{2}, v_{4}, v_{5}, v_{7}, v_{8}\}$ is a gp-set of $H_{3}$. Thus, $\Z(H_{3})=2<6 = \gp(H_{3})$. On the other hand, it can be seen that $\{u_{1}, \ldots, u_{s}, x, y, z, w, v_{1}, \ldots, v_{t-1}\}$ is a minimum zero forcing set of $H_{4}$, and that $\{u_{1}, \ldots, u_{s}, v_{1}, \ldots, v_{t}\}$ is a gp-set of $H_{4}$. Thus, $\Z(H_{4})=s+t+3>s+t = \gp(H_{4})$. So the zero forcing number and the general position number are not comparable on quasi-trees. But we do have the following result.

\begin{theorem}
If $G$ is a quasi-tree in which one of the following conditions hold:
\begin{enumerate}
\item[(i)] $G$ contains no pendent vertices,
\item[(ii)] $G$ contains a quasi-vertex $x$ such that $x$ does not have degree $2$ neighbors,
\end{enumerate}
then $\gp(G)\geq \Z(G)$.
\end{theorem}

\begin{proof}
If $G$ is itself a tree, then the result holds by~\eqref{eq:starting-point}. Hence assume in the rest that $G$ has at least one cycle.

(i) Suppose that $G$ contains no pendent vertices. Let $x$ be an arbitrary quasi-vertex of $G$. Then $G-x$ is a tree and by the assumption we see that $x$ must be adjacent to all the leaves of $G-x$, that is, $\mathcal{L}(G-x)\subseteq N_{G}(x)$. We claim that $\mathcal{L}(G-x)$ is a general position set of $G$. Indeed, if $u$ and $v$ are arbitrary leaves from $\mathcal{L}(G-x)$, then $d_G(u,v) = 2$ as each of them is adjacent to $x$. So $\mathcal{L}(G-x)$ is a set of vertices that are pairwise at distance $2$ and hence forms a general position set. We can now estimate as follows:
\begin{align*}
Z(G)& \leq \Z(G-x) + 1 \\
    & \leq (\gp(G-x) - 1) + 1 = \gp(G-x) \\
    & = \ell(G-x) \\
    & \leq \gp(G)\,.
\end{align*}
The first inequality follows by~\cite[Theorem 2.3]{Edholm}, the second inequality by~\eqref{eq:starting-point}, the equality follows because $G-x$ is a tree, while the last inequality follows by the argument above.

(ii) Suppose that $G$ contains a quasi-vertex $x$ such that no neighbor of $x$ is of degree $2$. This means that $x$ is adjacent to no leaf of $G-x$ which in turn implies that $\mathcal{L}(G-x) = \mathcal{L}(G)$. Since the set of leaves is a general position set in any graph, this means that $\gp(G-x)\leq \gp(G)$. Then, similarly as in (i), we conclude that $\Z(G) \leq \Z(G-x) + 1\leq\gp(G-x) = \ell(G-x) \leq \gp(G)$.
\end{proof}

We conclude the section with the following extension of~\eqref{eq:starting-point}.

\begin{proposition}
If $F$ is a forest with $k\ge 1$ non-trivial components, then $\gp(F)\geq \Z(F)+k$.
\end{proposition}

\begin{proof}
Let $T$ be a forest, let $T_{1}, \ldots, T_{k}$ be its non-trivial components, and let $x_{1}, \ldots, x_{s}$ be its isolated vertices, where $s\ge 0$. By~\eqref{eq:starting-point} we have $\gp(T_{i})\geq \Z(T_{i})+1$ for each $i \in [k]$. Let $S_{i}$ be a minimum forcing set of $T_{i}$ for
each $i\in [k]$. Then $S_{1}\cup \cdots \cup S_{k}\cup \{x_{1},\ldots,x_{s}\}$ is a minimum forcing set of $F$ and hence $\Z(F)=|S_{1}|+\cdots+|S_{k}|+s=\Z(T_{1})+\cdots+\Z(T_{k})+s\leq (\gp(T_{1})-1)+\cdots+(\gp(T_{k})-1)+s =\gp(T_{1})+\cdots+\gp(T_{k})-k+s$. For  each $i\in [k]$, let $R_i$ be a $\gp$-set of $T_{i}$. Then $R_{1}\cup \cdots \cup R_{k}\cup \{x_{1},\ldots,x_{s}\}$ is a gp-set of $F$. So, $\gp(T_{1})+\cdots+\gp(T_{k})+s=|R_{1}\cup \cdots \cup R_{k}\cup \{x_{1}, \ldots, x_{s}\}|\leq \gp(F)$. We conclude that $\Z(F)\leq \gp(F)-k$.
\end{proof}

\section{Concluding remarks}
\label{sec:conclude}

We have proved that the zero forcing number (plus maybe 1) is a lower bound for the general position number for trees, unicyclic graphs, block graphs, and special quasi-trees. We have also demonstrated that this does not hold for bicyclic graphs and for quasi-trees, hence we pose the following two problems.

\begin{problem}
Determine the bicyclic graphs $G$ such that $\gp(G)\geq \Z(G)$.
\end{problem}

\begin{problem}
Determine the quasi-trees $G$ such that $\gp(G)\geq \Z(G)$.
\end{problem}

Note that the graphs $H_1$ from Fig.~\ref{fig:bicyclic} and the graphs $H_4$ from Fig.~\ref{fig:quasi-trees} are bipartite. As we have seen, $\Z(H_1) > \gp(H_1)$ and $\Z(H_4) > \gp(H_4)$. On the other hand, \eqref{eq:starting-point} asserts that $\gp(T) > \Z(T)$ holds for trees $T$. Hence the following problem is also relevant.

\begin{problem}
Determine the bipartite graphs $G$ such that $\gp(G)\geq \Z(G)$.
\end{problem}

\section*{Acknowledgements}

Hongbo Hua was supported by National Natural Science Foundation of China under Grant No. 11971011. Sandi Klav\v{z}ar acknowledges the financial support from the Slovenian Research Agency (research core funding P1-0297, and projects N1-0095, J1-1693, J1-2452).

\section*{Competing Interests}

The authors have no relevant financial or non-financial interests to disclose.

\section*{Data Availability Statements}

All data generated or analysed during this study are included in this published article (and its supplementary information files).


\begin{thebibliography}{99}
{\small
\bibitem{SGWGroup}
AIM Minimum Rank - Special Graphs Work Group (F. Barioli, W. Barrett, S. Butler, S.M. Cioab\v{a}, D. Cvetkovi\'{c}, S.M. Fallat, C. Godsil, W. Haemers, L. Hogben, R. Mikkelson, S.
Narayan, O. Pryporova, I. Sciriha, W. So, D. Stevanovi\'{c}, H. van der Holst, K. Vander Meulen,
A.W. Wehe). Zero forcing sets and the minimum rank of graphs, Linear Algebra Appl.\ 428 (2008) 1628--1648.

\bibitem{Aazami}
A. Aazami, Hardness results and approximation algorithms for some problems on graphs (Ph.D. thesis), University of Waterloo, 2008.

\bibitem{Anand}
B.S. Anand, S.V. Ullas Chandran, M. Changat, S. Klav\v{z}ar, E.J. Thomas, Characterization of general position sets and its applications to cographs and bipartite graphs, Appl.\ Math.\ Comput.\ 359 (2019) 84--89.

\bibitem{Barioli}
F. Barioli, S. Fallet, L. Hogben, On the difference between the maximum multiplicity and path cover number for tree-like graphs,
Linear Algebra Appl.\ 409 (2005) 13--31.

\bibitem{Burgarth}
D. Burgarth, V. Giovannetti, Full control by locally induced relaxation, Phys.\ Rev.\ Lett.\ 99 (2007) 100501.

\bibitem{Dreyer}
P.A. Dreyer, F.S. Roberts, Irreversible $k$-threshold processes: Graph-theoretic threshold models of the spread of disease and of opinion, Discrete
Appl.\ Math.\ 157 (2009) 1615--1627.

\bibitem{Edholm}
C.J. Edholm, L. Hogben, M. Hyunh, J. LaGrange,  D.D. Row, Vertex and edge spread of zero forcing number, maximum nullity, and minimum rank of a graph, Linear Algebra Appl.\ 436 (2012) 4352--4372.

\bibitem{Eroh1}
L. Eroh, C.X. Kang, E. Yi,
Metric dimension and zero forcing number of two families of line graphs,
Math. Bohem. 139 (2014) 467--483.

\bibitem{Eroh2}
L. Eroh, C.X. Kang,  E. Yi, A comparison between the metric dimension and
zero forcing number of trees and unicyclic graphs, Acta Math.\ Sin.\ (Engl.\ Ser.)   33 (2017) 731--747.

\bibitem{Ferrero}
D. Ferrero, T. Kalinowski, S. Stephen, Zero forcing in iterated line digraphs, Discrete Appl.\ Math.\ 255 (2019) 198--208.

\bibitem{Ghorbani}
M. Ghorbani, S. Klav\v{z}ar, H.R. Maimani, M. Momeni, F. Rahimi-Mahid, G. Rus, The general position problem on Kneser graphs and on some graph
operations, Discuss.\ Math.\ Graph Theory 41 (2021) 1199--1213.

\bibitem{Hogben}
L. Hogben, Minimum rank problems, Linear Algebra Appl.\ 432 (2010) 1961--1974.

\bibitem{Javaid}
I. Javaid, I. Irshad, M. Batool, Z. Raza, On the zero forcing number of corona and lexicographic product of graphs, arXiv:1607.04071 [math.CO] (14 Jul 2016).

\bibitem{Kang1}
C.X. Kang, E. Yi, On zero forcing number of functigraphs, arXiv:1204.2238v2 [math.CO] (6 May 2012).

\bibitem{Kang2}
C.X. Kang, E. Yi, A comparison between the zero forcing number and the strong metric dimension of graphs, Lecture Notes Comp.\ Sci.\ 8881 (2014)  356--365.

\bibitem{Klavzar-Kuziak}
S. Klav\v{z}ar, D. Kuziak, I. Peterin, I.G. Yero, 
A Steiner general position problem in graph theory,
Comput.\ Appl.\ Math.\ 40 (2021) 223. 

\bibitem{Klavzar2}
S. Klav\v{z}ar, G. Rus, The general position number of integer lattices, Appl.\ Math.\ Comput.\ 390 (2021) 125664.

\bibitem{Klavzar1}
S. Klav\v{z}ar, I.G. Yero, The general position problem and resolving graphs, Open Math.\ 17 (2019) 1126--1135.

\bibitem{korner-1995}
  J.~K\"orner,
  On the extremal combinatorics of the Hamming space,
  J.\ Combin.\ Theory Ser A 71 (1995) 112--126.

\bibitem{Lu}
L. Lu, B. Wu, Z. Tang,
Proof of a conjecture on the zero forcing number of a graph, Discrete Appl.\ Math.\ 213 (2016) 233--237.

\bibitem{Manuel}
P. Manuel, S. Klav\v{z}ar, A general position problem in graph theory, Bull.\ Aust.\ Math.\ Soc.\ 98 (2018) 177--187.

\bibitem{neethu-2021}
  P.K.~Neethu, S.V.~Ullas Chandran, M.~Changat, S.~Klav\v{z}ar,
  On the general position number of complementary prisms,
  Fund.\ Inform.\ 178 (2021) 267--281.

\bibitem{Oboudi}
M.R. Oboudi, On the zero forcing number of trees, Iran J.\ Sci.\ Technol.\ Trans.\ Sci.\ 45 (2021) 1065--1070.

\bibitem{Patkos}
B. Patk\'{o}s, On the general position problem on Kneser graphs, Ars Math.\ Contemp.\ 18 (2020) 273--280.

\bibitem{Row}
D.D. Row, A technique for computing the zero forcing number of a graph
with a cut-vertex, Linear Algebra Appl.\ 436 (2012) 4423--4432.

\bibitem{thomas-2020}
  E.J.~Thomas, S.V.~Ullas Chandran,
  Characterization of classes of graphs with large general position number,
  AKCE Int.\ J.\ Graphs Comb.\ 17 (2020) 935--939.

\bibitem{Tian}
J. Tian, K. Xu, The general position number of Cartesian products involving
a factor with small diameter, Appl.\ Math.\ Comput.\ 403 (2021) 126206.

\bibitem{tian-2021}
  J.~Tian, K.~Xu, S.~Klav\v{z}ar,
  The general position number of Cartesian product of two trees,
  Bull.\ Aust.\ Math.\ Soc.\  104 (2021) 1--10.

\bibitem{UllasChandran}
S.V. Ullas Chandran, G. Jaya  Parthasarathy, The geodesic irredundant sets in graphs, Int.\ J.\ Math.\ Combin.\ 4 (2016) 135--143.

\bibitem{Yang}
B. Yang, Fast-mixed searching and related problems on graphs, Theoret.\ Comput.\ Sci.\ 507 (2013) 100--113.
}
\end{thebibliography}
\end{document}